\documentclass[11pt]{amsart}

\usepackage{amsmath,amssymb}
\usepackage{accents}
\usepackage{graphicx}
\usepackage{verbatim}
\usepackage{tikz}

\usetikzlibrary[topaths]

\newlength{\dhatheight}

\newtheorem{theorem}{Theorem}[section]

\newtheorem{proposition}[theorem]{Proposition}

\newtheorem{corollary}[theorem]{Corollary}
\newtheorem{definition}[theorem]{Definition}

\newtheorem{question}[theorem]{Question}

\newcommand{\Z}{{\mathbb Z}}
\newcommand{\Q}{{\mathbb Q}}

\begin{document}
\title[Stair step bridge spectrum does not imply high distance]
{Stair step bridge spectrum does not imply high distance}

\author[N.~Owad]{Nicholas Owad}
\address{Department of Mathematics\\
        University of Nebraska\\
         Lincoln NE 68588-0130, USA}
\email{nowad2@math.unl.edu}
\thanks{  Research partially funded through the MCTP Trainee program}

\begin{abstract}
Tomova, along with results of Bachman and Schleimer, showed that any high distance knot has a stair-step bridge spectrum. In this paper, we compute the bridge spectra and distance of generalized Montesinos knots.  In particular, we produce the first example of a class of knots which attain the stair-step bridge spectra but are not high distance, by computing the bridge spectra of certain pretzel knots and Montesinos knots. \end{abstract}

\maketitle



\section{Introduction}\label{sec:intro}




A $(g,b)$-splitting of a knot $K$ is a Heegaard splitting of $S^3=V_1 \cup_\Sigma V_2$ such that the genus of $V_i$ is $g$, $\Sigma$ intersects $K$ transversally, and $V_i \cap K$ is a collection of $b$ trivial arcs for $i=1,2$.  We define {\em the genus $g$ bridge number}  $b_g(K)$ to be the minimum $b$ for which a $(g,b)$-splitting exists.  The bridge spectrum of a knot $K$ is a strictly decreasing list of integers first defined by Zupan in \cite{Zupan}.  We obtain this list by first looking at $(g,b)$ -splittings of a knot in $S^3$. The {\em bridge spectrum} for $K$, which we denote ${\bf b}(K)$ is the list 

$$(b_0(K),b_1(K),b_2(K),\ldots ).$$

One should note that $b_0$ is the classical bridge number, $b(K)$, first defined by Schubert \cite{Schubert}.  The fact that bridge spectrum is strictly decreasing follows from a process called {\em meridional stabilization} where, given a $(g,b)$-splitting, we can create a $(g+1,b-1)$-splitting. Thus, every bridge spectrum is bounded above by the following sequence, which we call a stair-step spectrum,

 $$\left(b_{0}\left(K\right),b_{0}\left(K\right)-1,b_{0}\left(K\right)-2,\ldots,2,1,0\right).$$  

An invariant of a bridge surface of a knot is the {\em distance}, where we are considering distance in terms of the curve complex of the bridge surface. Tomova showed in \cite{Tomova}, with results of Bachman and Schleimer from \cite{BS}, the following theorem, as stated by Zupan in \cite{Zupan}:

\smallskip

\begin{theorem}\label{thm:dist}
Suppose $K$ is a knot in $S^3$ with a $(0,b)$-bridge sphere $\Sigma$ of sufficiently high distance (with respect to $b$). Then any $(g',b')$-bridge surface $\Sigma ' $  satisfying $b' = b_{g'} (K)$ is the result of meridional stabilizations performed on $\Sigma$. Thus
 $$\left(b_{0}\left(K\right),b_{0}\left(K\right)-1,b_{0}\left(K\right)-2,\ldots,2,1,0\right).$$  
 
\end{theorem}

A natural question to ask is the following:

\begin{question}
Given a knot, $K$, if the bridge spectrum of $K$ is stair-step, does $K$ necessarily have a $(0,b)$-bridge sphere that is high distance?
\end{question}

We will answer this question by showing that pretzel knots $K_n (p_1,p_2,\ldots,p_n)$ where $n\geq4$ are distance one, and that pretzel knots with $\gcd(p_1,p_2,\ldots,p_n)\not =1$ are stair-step.  

Theorem \ref{thm:dist} tells us that a generic knot has stair-step bridge spectrum.  We say there is a {\em gap} at index $g$ in the bridge spectrum if $b_g(K)<b_{g-1}(K)-1$.  Most work regarding bridge spectra focuses on finding bridge spectra with gaps.

Lustig and Moriah \cite{LM} define generalized Montesinos knots, which include Montesinsos knots and pretzel knots and show that a class of these generalized Montesinos knots have the property $t(K)+1=b(K)=b_0(K)$.  For a $(g,b)$-splitting, with $b>0$, we have $t(K) \leq g+b-1$, where $t(K)$ is the tunnel number, which comes to us from Morimoto, Sakuma, and Yokota \cite{MSY}.  From these two facts, one can quickly conclude that this class of generalized Montesinos knots has the stair-step bridge spectra.  We also show that this class of knots has distance one, so they do not satisfy the hypothesis of Tomova's theorem.

In section 2, we define the relevant terms and give some basic results.  In section three, we state and prove the main results and ask some questions.






\section{Notation and Background} \label{sec:notation}


\subsection{Bridge Spectrum}\label{subsec:bridgespectrum}

A knot $K$ in $S^3$ is an embedded copy of $S^1$.  Two knots are equivalent if there exists an ambient isotopy of $S^3$ taking one knot to the other.  Given a 3-manifold $M$ with boundary, a {\em trivial arc} is a properly embedded arc $\alpha$ ($\partial \alpha \subset \partial M$ ) that cobounds a disk $D$ with an arc $\beta \subset \partial M$, i.e., $\alpha \cap \beta =\partial \alpha =\partial \beta$ and there is an embedded disk $D \subset M$ such that $\partial D = \alpha \cup \beta$. We call the disk $D$ a {\em bridge disk}.  A {\em bridge splitting} of a knot $K$ in $S^3$ is a decomposition of $(S^3,K)$ into $(V_1,A_1) \cup_\Sigma (V_2,A_2)$, where $V_i$ is a handlebody with boundary $\Sigma$ and $A_i \subset V_i$ is a collection of trivial arcs for $i=1,2$.  One should note that when $K=\emptyset$, a bridge splitting is a {\em Heegaard splitting}.   A {\em $(g,b)$-splitting} is a bridge splitting with $g(V_i)=g$ and $|A_i|=b$ for $i=1,2$.

\begin{definition} \label{def:genus $g$ bridge number}

The genus $g$ bridge number of a knot $K$, $b_g(K)$, is the minimum $b$ such that a $(g,b)$-splitting of $K$ exists. We require that the knot can be isotoped into the genus $g$ Heegaard surface for $b_g (K)=0$. 

\end{definition}

As mentioned earlier, the genus zero bridge number is the classical bridge number. 

\begin{definition} \label{def:bridge spectrum}

The bridge spectrum of a knot $K$, ${\bf b}(K)$ is the list of genus $g$ bridge numbers:

$$(b_0(K),b_1(K),b_2(K),\ldots ).$$

\end{definition}

A simple closed curve in the boundary of a handlebody is called {\em primitive} if it transversely intersects the boundary of a properly embedded essential disk on the boundary in a single point.  Some define $b_g(K)=0$ when $K$ can be isotoped into a genus $g$ Heegaard surface and is primitive.  We will define another invariant, $\hat{b}_g(K)$ with the following added requirement. 

\begin{definition} \label{def:primitive genus $g$ bridge number}

The (primitive) genus $g$ bridge number of a knot $K$, $\hat{b}_g(K)$, is the minimum $b$ such that there is a $(g,b)$-splitting of $K$ exists. We require that the knot can be isotoped into the genus $g$ Heegaard surface and is primitive for $\hat{b}_g (K)=0$.

\end{definition}

\begin{definition} \label{def:primitive bridge spectrum }

The (primitive) bridge spectrum of a knot $K$, $ {\bf \hat{b}}(K)$ is the list of genus $g$ bridge numbers:

$$(\hat{b}_0(K),\hat{b}_1(K),\hat{b}_2(K),\ldots ).$$

\end{definition}

The only potential difference between these two spectra is the last non-zero term in the sequence. For example, a non-trivial torus knot $K=T_{p,q}$, we have ${\bf b}(K)=(\min\{p,q\},0)$ while ${\bf \hat{b}}(K)=(\min\{p,q\},1,0)$.  So they are distinct invariants. But there are knots for which they coincide. Any 2-bridge knot that is not a torus knot will have ${\bf \hat{b}}(K)={\bf b}(K)=(2,1,0)$.  

Bridge spectra are always strictly decreasing sequences.  To see this, take any $(g,b)$-splitting of a knot.  Then by definition we have a decomposition as  $(V_1,A_1) \cup_\Sigma (V_2,A_2)$.  Let $N(\cdot)$ be a closed regular neighborhood and $\eta(\cdot)$ an open regular neighborhood. Intuitively, take any trivial arc $\alpha$ in, say, $(V_1,A_1)$, then $N(\alpha)$ is a closed neighborhood, we take this neighborhood from $V_1$ and give to $V_2$, which produces one higher genus handlebodies, and one less trivial arc. More carefully, let $W_1=(V_1-\eta(\alpha))$ and $W_2=(V_2 \cup N(\alpha))$.  Notice that $g(W_i)=g(V_i)+1$ and $B_1 = A_1 - \{\alpha\}$, $B_2 = A_2 \cup \alpha$ is $A_2$ with two arcs combined into a single arc by connecting them with $\alpha$.  Thus, we have  $(W_1,B_1) \cup_{\Sigma'} (W_2,B_2)$ which is a $(g+1,b-1)$-splitting. This process is called {\em meridional stabilization} and gives us the relation $b_{g+1}(K) \leq b_g(K)-1$.  
Hence, every bridge spectrum is bounded above by the stair-step spectrum

 $$\left(b_{0}\left(K\right),b_{0}\left(K\right)-1,b_{0}\left(K\right)-2,\ldots,2,1,0\right).$$  

For more background on bridge spectra of knots, see \cite{BTZ}, \cite{doll}, \cite{MS}, and \cite{Zupan}.


\subsection{Tunnel Number}\label{subsec:tunnel}

For a knot, $K$ in $S^3$, we define the exterior of $K$ as $S^3 - \eta(K)$ and denote it by $E(K)$. A family of mutually disjoint properly embedded arcs $\Gamma$ in $E(K)$ is said to be an {\em unknotting tunnel system} if $E(K)-\eta(\Gamma)$ is homeomorphic to a handlebody. 

\begin{definition}

The tunnel number of a knot $K$, $t(K)$, is the minimum size over all unknotting tunnel systems for $K$. 

\end{definition}
 
Morimoto gives an equivalent definition in \cite{Morimoto}, which we will use here. There is a Heegaard splitting $(V_1, V_2)$ of $S^3$ such that a handle of $V_1$ contains $K$ as a core of $V_1$.

\begin{definition}

The minimum genus of $V_1$ minus one over all Heegaard splittings satisfying the above fact is the tunnel number, $t(K)$. 

\end{definition}


\subsection{Distance}\label{subsec:dist}

For a more thorough discussion on distance, see \cite{Tomova}. Given a compact, orientable, properly embedded surface $S$ in a 3-manifold $M$, the 1-skeleton of the {\em curve complex}, $\mathcal{C}(S)$, is the graph whose vertices correspond to isotopy classes of essential simple closed curves in $S$ and two vertices are connected if the corresponding isotopy classes have disjoint representatives. For two subsets $A$ and $B$ of $\mathcal{C}(S)$, we define the distance between them, $d(A,B)$ to be the length of the  shortest path from $A$ to $B$.  

For any subset $X\subset S^3$, we define $X_K$ to be $E(K)\cap X$.

\begin{definition}

Suppose $M$ is a closed, orientable irreducible 3-manifold containing a knot $K$ and suppose $P$ is a bridge surface for $K$ splitting $M$ into handlebodies $A$ and $B$. The curve complex $C(P_K)$ is a graph with vertices corresponding to isotopy classes of essential simple closed curves on $P_K$. Two vertices are adjacent in $C(P_K)$ if their corresponding classes of curves have disjoint representatives. Let $\mathcal{A}$ (resp $\mathcal{B}$) be the set of all essential simple closed curves on $P_K$ that bound disks in $A_K$ (resp. $B_K$). Then $d(P, K) = d(\mathcal{A}, \mathcal{B})$ measured in $C(P_K )$.

\end{definition}


\subsection{Generalized Montesinos Knots}\label{subsec:GMK}

J. Montesinos defined the class of knots and links that now bear his name in 1973 in \cite{Montesinos}.  Given a rational number $\frac{\beta}{\alpha} \in \Q$, there is a unique continued fraction decomposition 
\[
\frac{\beta}{\alpha}=   a_1+\cfrac{1}{a_2+\cfrac{1}{a_3+\cfrac{1}{\ddots +\cfrac{1}{a_m } }}} = :[a_1,a_2, \ldots, a_m]
\]
where $a_i \not = 0$ for all $i=1,\ldots,m$ and $m$ is odd. For each rational number, there is an associated tangle that we can think of as two unoriented arcs on a sphere with fixed endpoints.  

\begin{figure}
\begin{tikzpicture}



\draw (-4.5,1.5) rectangle +(2,2);

\draw (-4,1.5) -- +(0,-1);
\draw (-3,1.5) -- +(0,-1);
\draw (-4,3.5) -- +(0,1);
\draw (-3,3.5) -- +(0,1);

\node at (-3.5,2.5) {$3,5$};

\draw [thick] (-2,2.625) -- +(.5,0);
\draw [thick] (-2,2.375) -- +(.5,0);

\draw [dashed] (-.25,-.25) rectangle +(3.5,5.5);

\draw [rounded corners] (.4,3.6) -- (0,4) -- (1,5) -- (1.4,4.6);
\draw [rounded corners] (.4,2.6) -- (0,3) -- (2,5) -- (2.4,4.6);
\draw [rounded corners] (0,-1) -- (0,0) -- (3,3) -- (2.6,3.4);
\draw [rounded corners] (.4,.6) -- (0,1) -- (3,4) -- (2.6,4.4);
\draw [rounded corners] (.4,1.6) -- (0,2) -- (3,5) -- (3,6);
\draw [rounded corners] (.6,.4) -- (1,0) -- (3,2) -- (2.6,2.4);
\draw [rounded corners] (1.6,.4) -- (2,0) -- (3,1) -- (2.6,1.4);

\draw (.6, 1.4) -- +(.3,-.3);
\draw (1.1, .9) -- +(.3,-.3);

\draw (.6, 2.4) -- +(.3,-.3);
\draw (1.1, 1.9) -- +(.3,-.3);
\draw (1.6, 1.4) -- +(.3,-.3);
\draw (2.1, .9) -- +(.3,-.3);
\draw [rounded corners] (2.6,.4) -- (3,0) -- (3,-1) ;

\draw (.6, 3.4) -- +(.3,-.3);
\draw (1.1, 2.9) -- +(.3,-.3);
\draw (1.6, 2.4) -- +(.3,-.3);
\draw (2.1, 1.9) -- +(.3,-.3);

\draw (.6, 4.4) -- +(.3,-.3);
\draw (1.1, 3.9) -- +(.3,-.3);
\draw (1.6, 3.4) -- +(.3,-.3);
\draw (2.1, 2.9) -- +(.3,-.3);
\draw [rounded corners] (.4,4.6) -- (0,5) -- (0,6);

\draw (1.6, 4.4) -- +(.3,-.3);
\draw (2.1, 3.9) -- +(.3,-.3);

\end{tikzpicture}

\caption{A rational tangle $\frac{3}{5}$.}

\end{figure}
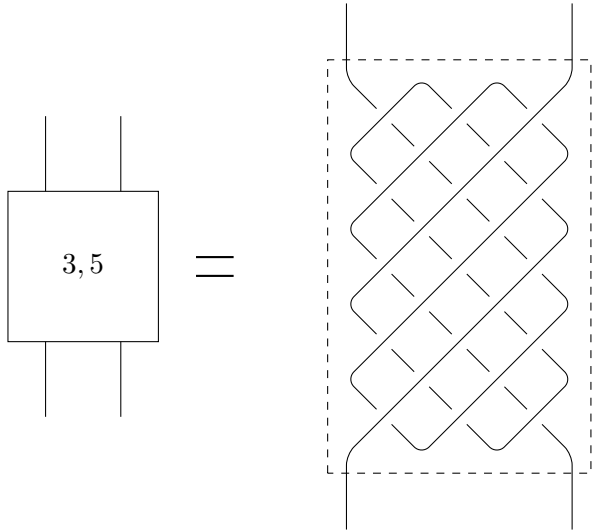

\begin{definition} 
A Montesinos knot or link $M(\frac{\beta_1}{\alpha_1},\frac{\beta_2}{\alpha_2},\ldots,\frac{\beta_n}{\alpha_n}\vert e)$ is the knot in Figure 3 where each $\beta_i,\alpha_i$ for $i=1,\ldots, n$ represents a rational tangle given by $\frac{\beta_i}{\alpha_i}$, and $e$ represents the number of positive half-twists. If $e$ is negative, we have negative half-twists instead. 

\end{definition}

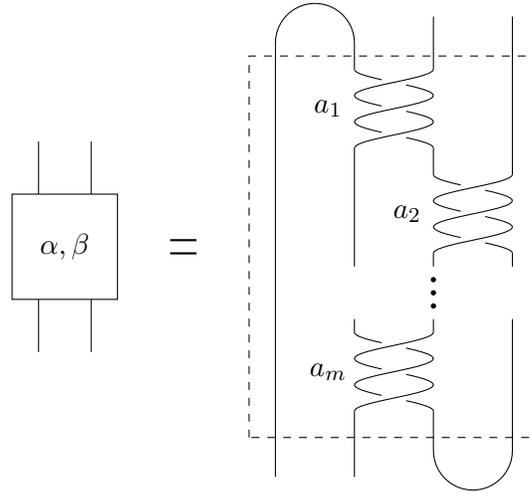
\begin{figure}
\begin{tikzpicture}[scale = .7]



\draw (0.5,4.625) rectangle +(2,2);

\draw (1,4.625) -- +(0,-1);
\draw (2,4.625) -- +(0,-1);
\draw (1,6.625) -- +(0,1);
\draw (2,6.625) -- +(0,1);

\node at (1.5,5.625) {$\alpha,\beta$};

\draw [thick] (3.5,5.5) -- (4,5.5);
\draw [thick] (3.5,5.75) -- (4,5.75);

\draw [dashed] (5,2) rectangle +(5.5,7.25);

\draw(5.5,1.25) -- (5.5, 9.5);

\draw(7,9) -- (7, 9.5);
\draw(7,5.25) -- (7, 7.5);
\draw(7,4) -- (7, 4.25);
\draw(7,1.25) -- (7, 2.5);

\draw(8.5,1.75) -- (8.5, 2.5);
\draw(8.5,4) -- (8.5, 4.25);
\draw(8.5,5.25) -- (8.5, 5.5);
\draw(8.5,7) -- (8.5, 7.5);
\draw(8.5,9) -- (8.5, 10);

\draw(10,1.75) -- (10, 4.25);
\draw(10,5.25) -- (10, 5.5);
\draw(10,7) -- (10, 10);

\draw [domain=0:180] plot ({6.25+.75*cos(\x)}, {9.5+.75*sin(\x)});	
\draw [domain=180:360] plot ({9.25+.75*cos(\x)}, {1.75+.75*sin(\x)});


\draw [domain=0:.4] plot ({.75*cos(pi* \x r )+7.75}, .5*\x+7.5);

\draw [domain=0:1.4] plot ({-.75*cos(pi* \x r )+7.75}, .5*\x+7.5);

\draw [domain=.6:2.4] plot ({.75*cos(pi* \x r )+7.75}, .5*\x+7.5);

\draw [domain=.6:2] plot ({.75*cos(pi* \x r )+7.75}, .5*\x+8);
\draw [domain=.6:1] plot ({.75*cos(pi* \x r )+7.75}, .5*\x+8.5);


\draw [domain=0:.4] plot ({.75*cos(pi* \x r )+9.25}, .5*\x+5.5);

\draw [domain=0:1.4] plot ({-.75*cos(pi* \x r )+9.25}, .5*\x+5.5);

\draw [domain=.6:2.4] plot ({.75*cos(pi* \x r )+9.25}, .5*\x+5.5);

\draw [domain=.6:2] plot ({.75*cos(pi* \x r )+9.25}, .5*\x+6);
\draw [domain=.6:1] plot ({.75*cos(pi* \x r )+9.25}, .5*\x+6.5);

\draw [domain=0:.4] plot ({.75*cos(pi* \x r )+7.75}, .5*\x+2.5);

\draw [domain=0:1.4] plot ({-.75*cos(pi* \x r )+7.75}, .5*\x+2.5);

\draw [domain=.6:2.4] plot ({.75*cos(pi* \x r )+7.75}, .5*\x+2.5);

\draw [domain=.6:2] plot ({.75*cos(pi* \x r )+7.75}, .5*\x+3);
\draw [domain=.6:1] plot ({.75*cos(pi* \x r )+7.75}, .5*\x+3.5);

\draw [fill] (8.5,4.5) circle [radius=0.04];	
\draw [fill] (8.5,4.75) circle [radius=0.04];	
\draw [fill] (8.5,5) circle [radius=0.04];

\node at (6.5,3.25) {$a_{m}$};
\node at (8,6.25) {$a_{2}$};
\node at (6.5,8.25) {$a_{1}$};

\end{tikzpicture}

\caption{A rational tangle in the 2-bridge form.}

\end{figure}

\begin{figure}
\begin{tikzpicture}[xscale = .7, yscale = .7]




\draw (9.5,3) rectangle +(2,2);

\draw (9.5,9) rectangle +(2,2);
\draw (9.5,12) rectangle +(2,2);

\draw [domain=0:.4] plot ({.5*cos(pi* \x r )+1.5}, \x+6);

\draw [domain=0:1.4] plot ({- .5*cos(pi* \x r )+1.5}, \x+6);

\draw [domain=.6:2.4] plot ({.5*cos(pi* \x r )+1.5}, \x+6);

\draw [domain=.6:2.4] plot ({.5*cos(pi* \x r )+1.5}, \x+7);

\draw [domain=.6:2] plot ({.5*cos(pi* \x r )+1.5}, \x+8);
\draw [domain=.6:1] plot ({.5*cos(pi* \x r )+1.5}, \x+9);

\draw [rounded corners] (1,6) -- (1,1) -- (11,1)--(11,3);
\draw [rounded corners] (2,6) -- (2,2) -- (10,2)--(10,3);

\draw [rounded corners] (1,10) -- (1,16) -- (11,16)--(11,14);
\draw [rounded corners] (2,10) -- (2,15) -- (10,15)--(10,14);

\foreach \a in {5,8,11}{
	\draw (10,\a) -- (10,\a+1);
	\draw (11,\a) -- (11,\a+1);
	}

\draw [fill] (10.5,6.7) circle [radius=0.05];	
\draw [fill] (10.5,7) circle [radius=0.05];	
\draw [fill] (10.5,7.3) circle [radius=0.05];	

\node at (0.5,8) {$e$};

\node at (10.5, 13) {$\beta_1,\alpha_1$};
\node at (10.5, 10) {$\beta_2,\alpha_2$};
\node at (10.5, 4) {$\beta_n,\alpha_n$};

\end{tikzpicture}
\caption{The Montesinos link $M(\frac{\beta_1}{\alpha_1},\frac{\beta_2}{\alpha_2},\ldots,\frac{\beta_n}{\alpha_n}\vert e)$}

\end{figure}
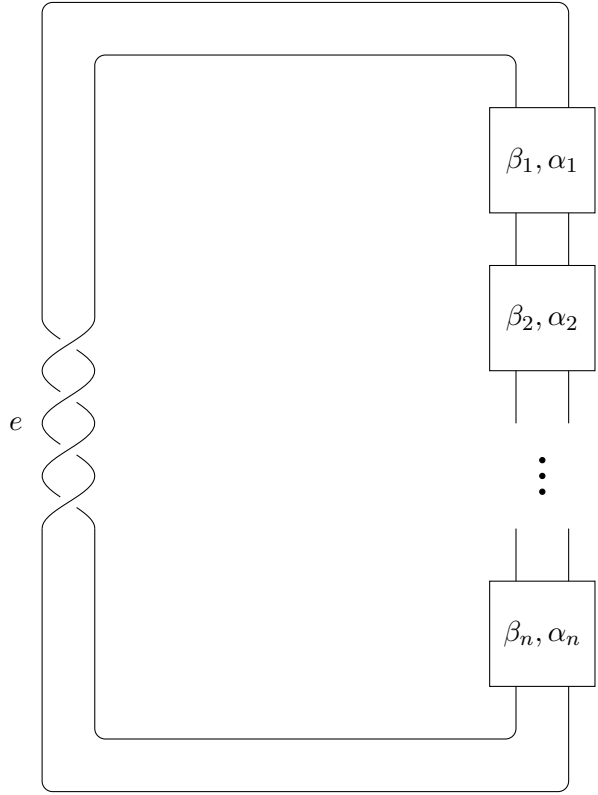

Lustig and Moriah in \cite{LM} defined the class of knots which we will describe in the rest of this section.  They based their definition off of Boileau and Zieschang \cite{BZ}, who prove that any Montesinos knot, which does not have integer tangles, has bridge number $n$. Consider figure \ref{fig:gmk}. Each $\alpha_{i,j},\beta_{i,j}$ is the 4-plat diagram from the rational tangle defined by $\alpha_{i,j}/\beta_{i,j}$ and each $B_j$ is a double of an $n$-braid. For a more in depth conversation, see \cite{BZ} and \cite{LM}. They exhibit a number diagrams which show that every Montesinos Knot and in particular, every pretzel knot, is a Generalized Montesinos Knot. One should note that for a pretzel knot $K_n=(p_1,\ldots, p_n)$, the corresponding rational tangles are $p_k=\alpha_{i,j}/\beta_{i,j}$.  Then define $\alpha=\text{gcd}(\alpha_{i,j}:i=1,\ldots,\ell, j=1,\ldots,m)$

\begin{definition} 
A Generalized Montesinos knot or link is the knot in Figure 4 where each $\beta_{i,j},\alpha_{i,j}$ for $i=1,\ldots, m$ and $j=1,\ldots, \ell$ represents a rational tangle given by $\frac{\beta_{i,j}}{\alpha_{i,j}}$, and $B_i$ represents a $2n$-braid, which is obtained by doubling an $n$-braid.

\end{definition}

The main result from Lustig and Moriah's paper that we will use here is the following:

\begin{theorem} \label{LMt} [ \cite{LM} theorem 0.1 ]
Let $K$ be a generalized Montesinos knot/link as in Figure 1 below, with $2n$-plats. Let  $\alpha=\text{gcd}(\alpha_{i,j}:i=1,\ldots,\ell, j=1,\ldots,m)$.  If $\alpha \not = 1$ then rk$(\pi_1(S^3-K))=t(K)+1=b(K)=n$.
\end{theorem}

\newcount \mycount

\begin{figure}\label{fig:gmk}
 
 \begin{tikzpicture}[transform shape]
 
 

\draw (0,3.5) rectangle +(13,2);
\draw (0,7.5) rectangle +(13,2);
\draw (0,12.5) rectangle +(13,2);


\node at (1,2) {$\alpha_{\ell,1},\beta_{\ell,1}$};
\node at (4,2) {$\alpha_{\ell,2},\beta_{\ell,2}$};
\node at (7,2) {$\alpha_{\ell,3},\beta_{\ell,3}$};
\node at (12,2) {$\alpha_{\ell,m},\beta_{\ell,m}$};

\node at (7,4.5) {$B_{\ell-1}$};
\node at (7,8.5) {$B_{2}$};
\node at (7,13.5) {$B_{1}$};

\foreach \d [count = \di] in {16,11}{
	\node at (12,\d) {$\alpha_{\di,m},\beta_{\di,m}$};
	
	\foreach \e [count = \ei] in {0,3,6}{
		\node at (\e+1,\d) {$\alpha_{\di,\ei},\beta_{\di,\ei}$};

		}
	}

\foreach \a [count = \ai] in {0,3,6,11}{
	\draw (\a,1) rectangle +(2,2);
	\draw (\a,10) rectangle +(2,2);
	\draw (\a,15) rectangle +(2,2);

	\draw [domain=180:360] plot ({\a+.35+.15*cos(\x)}, {.5+.15*sin(\x)});	
	\draw [domain=180:360] plot ({\a+1.65+.15*cos(\x)}, {.5+.15*sin(\x)});

	\draw [domain=0:180] plot ({\a+.35+.15*cos(\x)}, {17.5+.15*sin(\x)});	
	\draw [domain=0:180] plot ({\a+1.65+.15*cos(\x)}, {17.5+.15*sin(\x)});

	\foreach \b in {.5,3,5.5,7,9.5,12,14.5,17}{
		\draw (\a+.2,\b) -- (\a+.2,\b+.5);
		\draw (\a+.5,\b) -- (\a+.5,\b+.5);

		\draw (\a+1.8,\b) -- (\a+1.8,\b+.5);
		\draw (\a+1.5,\b) -- (\a+1.5,\b+.5);
	}
	
	}
\foreach \c in {2,11,16}{
	
	\draw [fill] (9,\c) circle [radius=0.05];	
	\draw [fill] (9.5,\c) circle [radius=0.05];	
	\draw [fill] (10,\c) circle [radius=0.05];	
	}
	
\draw [fill] (7,6.2) circle [radius=0.05];	
\draw [fill] (7,6.5) circle [radius=0.05];	
\draw [fill] (7,6.8) circle [radius=0.05];


\end{tikzpicture}

\caption{A Generalized Monetesinos Knot}

\end{figure}
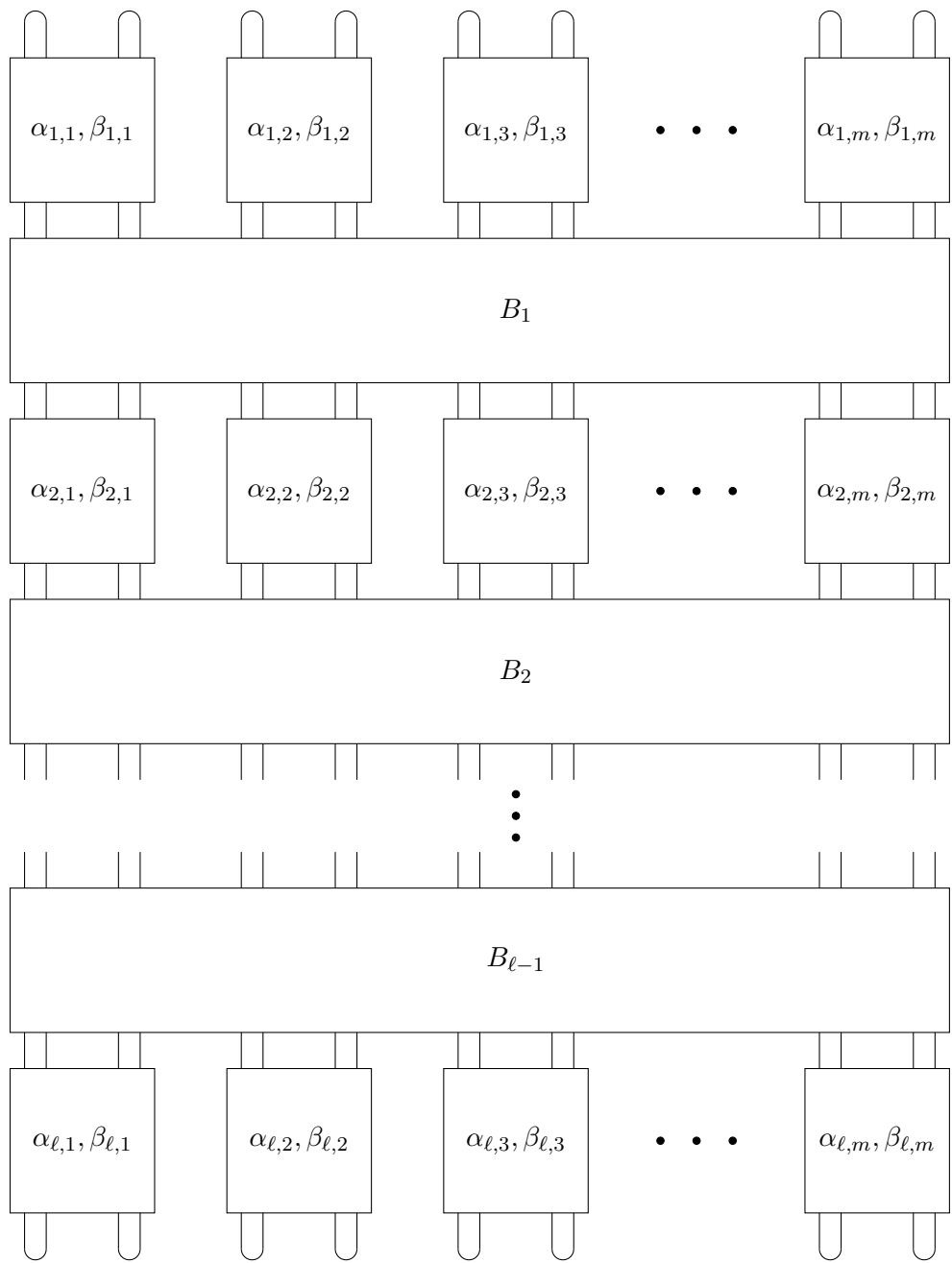
 




\section{Results}\label{sec:results}


We first begin with some basic results that are used many places. The first proposition appears in \cite{MSY} without proof.  We provide one here for completeness. 

\begin{proposition} \label{prop:MSYsprop}
Given a knot with a $(g,b)$-splitting with $b>0$, or if $b=0$ and the knot is primitive, we have the following relation: $t(K)\leq g+b-1$.

\end{proposition}

\begin{proof}

When $b>0$, using the definition in \cite{Morimoto}, we can meridionally stabilize each of the $b$ arcs.  This takes the genus $g$ surface and adds in $b$ more handles, which gives us a $g+b$ genus surface. This is precisely the situation in Morimoto's definition. Hence, the tunnel number is at most $g+b-1$. If $b=0$ and the knot is primitive, then there is an essential disk which intersects the knot exactly once. This essential disk  must be a meridian disk and thus, our knot must be a tunnel number at most $g-1$. 
\end{proof}

This is the driving force behind the following well known proposition:

\begin{proposition} \label{prop:stairstep}
Given a knot $K$ with $t(K)+1=b(K)=b_0(K)$,  $K$ has the stair-step primitive bridge spectrum, ${\bf \hat{b}}(K)=(\hat{b}_0(K),\hat{b}_0(K)-1,\ldots,2,1,0)$.
\end{proposition}

\begin{proof}

By proposition \ref{prop:MSYsprop}, for any $(g,b)$-splitting,  $t(K)\leq g+b -1$, and thus, $\hat{b}_0(K)-1\leq g+b-1$.  As stated earlier that every spectrum is bounded above by the stair-step spectrum. Hence, $g+b\leq \hat{b}_0(K) \iff g+b-1\leq \hat{b}_0(K)-1$.
\end{proof}

\begin{theorem}\label{thm:bsgmk}
A generalized Montesinos knot or link, $K$ with $\alpha \not = 1$ has the stair-step primitive bridge spectrum, ${\bf \hat{b}}(K)=(\hat{b}_0(K),\hat{b}_0(K)-1,\ldots,2,1,0)$.

\end{theorem}
 
\begin{proof} This is immediate from the previous proposition and theorem \ref{LMt}.
\end{proof}

Notice that theorem \ref{thm:bsgmk} implies any pretzel knot $K_n=K_n(p_1,\ldots,p_n)$ also has the primitive stair step bridge spectrum if $\alpha=\text{gcd}(p_1,\ldots,p_n)\not = 1$.  It should be noted that pretzel knots $K_n$ with $\alpha \not = 1$, while having stair-step primitive bridge spectra, have the bridge spectra ${\bf b}(K_n)=(n,n-1,\ldots, 3,2,0)$. This is because any pretzel knot embeds into a genus $n-1$ surface. See figure~\ref{fig:pretz_embed}.  These facts prove the following corollary.

\begin{corollary}\label{cor:bsp}
Given a $K_n=K_n(p_1,\ldots,p_n)$ pretzel knot with $\gcd(p_1,\ldots,p_n)\not = 1$ the primitive bridge spectrum is stair-step, i.e. ${\bf \hat{b}}(K_n)=(n,n-1,\ldots,2,1,0)$ and ${\bf b}(K_n)=(n,n-1,\ldots,3,2,0)$.
\end{corollary}

\begin{question}
Are pretzel knots exactly the class of Montesinos knots which have the property that ${\bf \hat{b}}(K)\not={\bf {b}}(K)$?
\end{question}


\begin{figure}\label{fig:pretz_embed}
\begin{tikzpicture}


\filldraw[fill=lightgray]
 (0,6) to [out=90,in=180] (2,8) -- (12,8) to [out=0, in=90] (14,6)--(14,2) to [out=270, in=0] (12,0) -- (2,0) to [out=180,in=270] (0,2) -- (0,6) --cycle
 
 \foreach \a in {2,6,10}{
(2+\a,5) to [out=90, in=0] (1+\a,6) to [out=180,in=90] (\a,5) -- (\a,3) to [out=270,in=180] (\a+1,2) to [out=0,in=270] (2+\a,3) -- (2+\a,5)--cycle};

\draw [thick] (2/3,2.5) to [out=270, in=180] (3,2/3) -- (11,2/3) to [out=0,in=270] (13+1/3,2.5);
\draw [thick] (2/3,5.5) to [out=90, in=180] (3,8-2/3) -- (11,8-2/3) to [out=0,in=90] (13+1/3,5.5);

\foreach \a in {0,4,8}{
\draw [thick] (4/3+\a,2.5) to [out=270, in=180] (3+\a,4/3) -- (3+\a,4/3) to [out=0,in=270] (5-1/3+\a,2.5);
\draw [thick] (4/3+\a,5.5) to [out=90, in=180] (3+\a,8-4/3) -- (3+\a,8-4/3) to [out=0,in=90] (5-1/3+\a,5.5);
}

\draw [thick] (4/3,2.5) to [out=90,in=270] (2,10/3);
\draw [thick, dashed] (2,10/3) to [out=90,in=270] (0,13/3);
\draw [thick] (0,13/3) to [out=90,in=270] (4/3,5.5);

\draw [thick] (2/3,2.5) to [out=90,in=270] (2,11/3);
\draw [thick, dashed] (2,11/3) to [out=90,in=270] (0,14/3);
\draw [thick] (0,14/3) to [out=90,in=270] (2/3,5.5);

\draw [thick] (16/3,2.5) to [out=90,in=270] (4,10/3);
\draw [thick, dashed] (4,10/3) to [out=90,in=270] (6,5.5-5/6);
\draw [thick] (6,5.5-5/6) to [out=90,in=270] (14/3,5.5);

\draw [thick] (14/3,2.5) to [out=90,in=270] (4,3);
\draw [thick, dashed] (4,3) to [out=90,in=270] (6,4-1/3);
\draw [thick] (6,4-1/3) to [out=90,in=270] (4,4+1/3);
\draw [thick,dashed] (4,4+1/3) to [out=90,in=270] (6,5);
\draw [thick] (6,5) to [out=90,in=270] (16/3,5.5);

\begin{scope}[xshift=4cm]
\draw [thick] (16/3,2.5) to [out=90,in=270] (4,10/3);
\draw [thick, dashed] (4,10/3) to [out=90,in=270] (6,5.5-5/6);
\draw [thick] (6,5.5-5/6) to [out=90,in=270] (14/3,5.5);

\draw [thick] (14/3,2.5) to [out=90,in=270] (4,3);
\draw [thick, dashed] (4,3) to [out=90,in=270] (6,4-1/3);
\draw [thick] (6,4-1/3) to [out=90,in=270] (4,4+1/3);
\draw [thick,dashed] (4,4+1/3) to [out=90,in=270] (6,5);
\draw [thick] (6,5) to [out=90,in=270] (16/3,5.5);
\end{scope}

\begin{scope}[xscale=-1,xshift=-18cm]
\draw [thick] (16/3,2.5) to [out=90,in=270] (4,10/3);
\draw [thick, dashed] (4,10/3) to [out=90,in=270] (6,5.5-5/6);
\draw [thick] (6,5.5-5/6) to [out=90,in=270] (14/3,5.5);

\draw [thick] (14/3,2.5) to [out=90,in=270] (4,3);
\draw [thick, dashed] (4,3) to [out=90,in=270] (6,4-1/3);
\draw [thick] (6,4-1/3) to [out=90,in=270] (4,4+1/3);
\draw [thick,dashed] (4,4+1/3) to [out=90,in=270] (6,5);
\draw [thick] (6,5) to [out=90,in=270] (16/3,5.5);
\end{scope}

\end{tikzpicture}

\caption{A pretzel knot $K_4(2,-3,-3,3)$ embedding on a genus three surface.}

\end{figure}
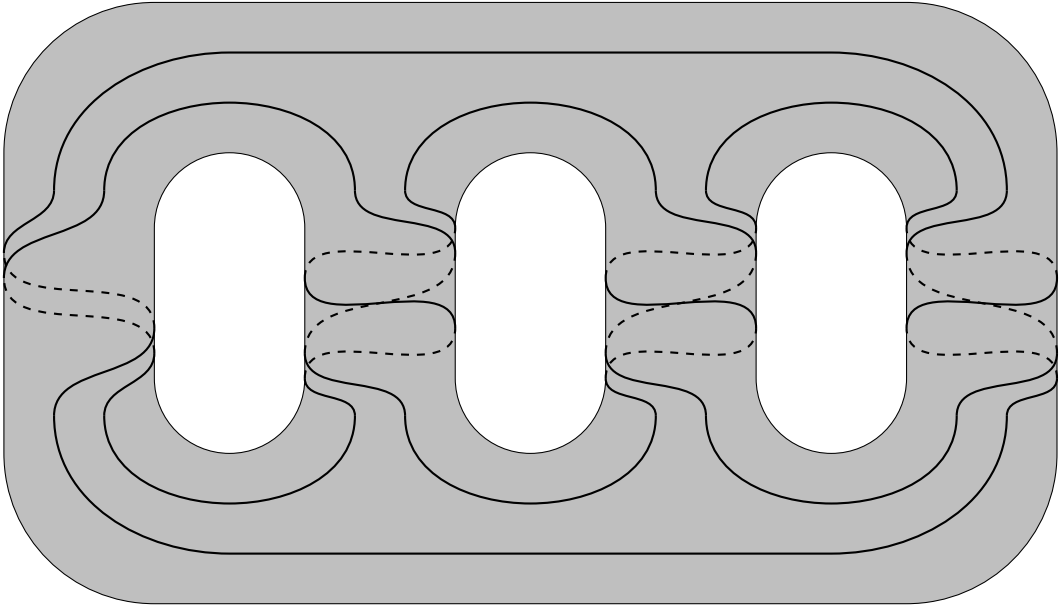


One should ask what happens when $\alpha=1$, to which we do not have a complete answer here. Morimoto, Sakuma, and Yokota in \cite{MSY} show that there are Montesinos knots $K$ which have $\alpha =1$ and $t(K)+2=b(K)$.  Here is their theorem as stated by Hirasawa and Murasugi in \cite{HM}:

\begin{theorem} \label{MSYtunnel}

The Montesinos knot $K=M(\frac{\beta_1}{\alpha_1},\frac{\beta_2}{\alpha_2},\ldots,\frac{\beta_r}{\alpha_r}\vert e)$ has tunnel number one if and only if one of the following conditions is satisfied.

\begin{enumerate}
\item $r=2.$
\item $r=3, \beta_2/\alpha_2 \equiv \beta_3/\alpha_3 \equiv \pm \frac{1}{3}$ in $\Q/ \Z$, and $e+\sum^3_{i=1}  \beta_i/\alpha_i = \pm 1/(3\alpha_1)$.
\item $r=3, \alpha_2$ and $\alpha_3$ are odd, and $\alpha_1=2$.

\end{enumerate}
\end{theorem}

Any pretzel link $K_3(p_1,p_2,2)$ is a knot if and only if $p_1$ and $p_2$ are odd.  Therefore, by theorem \ref{MSYtunnel}, every pretzel knot of this form will be tunnel number 1 but bridge number 3.  Thus, we cannot use proposition \ref{prop:stairstep} to compute the bridge spectrum. It is unknown if there is a pretzel knot or Montesinos knot which has bridge spectrum $(3,1,0)$.

Another point of interest is that any $n$-pretzel knot $K_n$ with $n\geq4$ is a distance one knot.

\begin{proposition} \label{prop:dist1}
If $K_n(p_1,\ldots,p_n)$ is a pretzel knot with $n\geq 4$,  and $P$ a genus zero bridge surface for $K_n$, then $d(P,K_n)=1$.
\end{proposition}

\begin{proof}

See figure~\ref{fig:pretzeldisks}.  The two disks that are above and below the bridge surface $P$ are bridge disks for the knot.  If we would thicken these disks by taking an $\epsilon$ neighborhood of each, and then consider the boundary of these neighborhoods, we obtain disks which bound essential simple closed curves in $P$. Since they are disjoint curves, the vertices they correspond to in the curve complex are adjacent. Hence, $d(P,K_n)\leq 1$.  If the distance is zero, we would have disks on opposite sides of $P$ which bound essential simple closed curves that are isotopic to each other, so we can isotope our disks to have the same boundary, and form a sphere.  This implies that our knot has at least two components, one on each side of the sphere, contradicting our assumption that $K_n$ is a knot. Hence, $d(P,K_n)=1$.

\begin{figure}
\begin{center}
\includegraphics[width=6in]{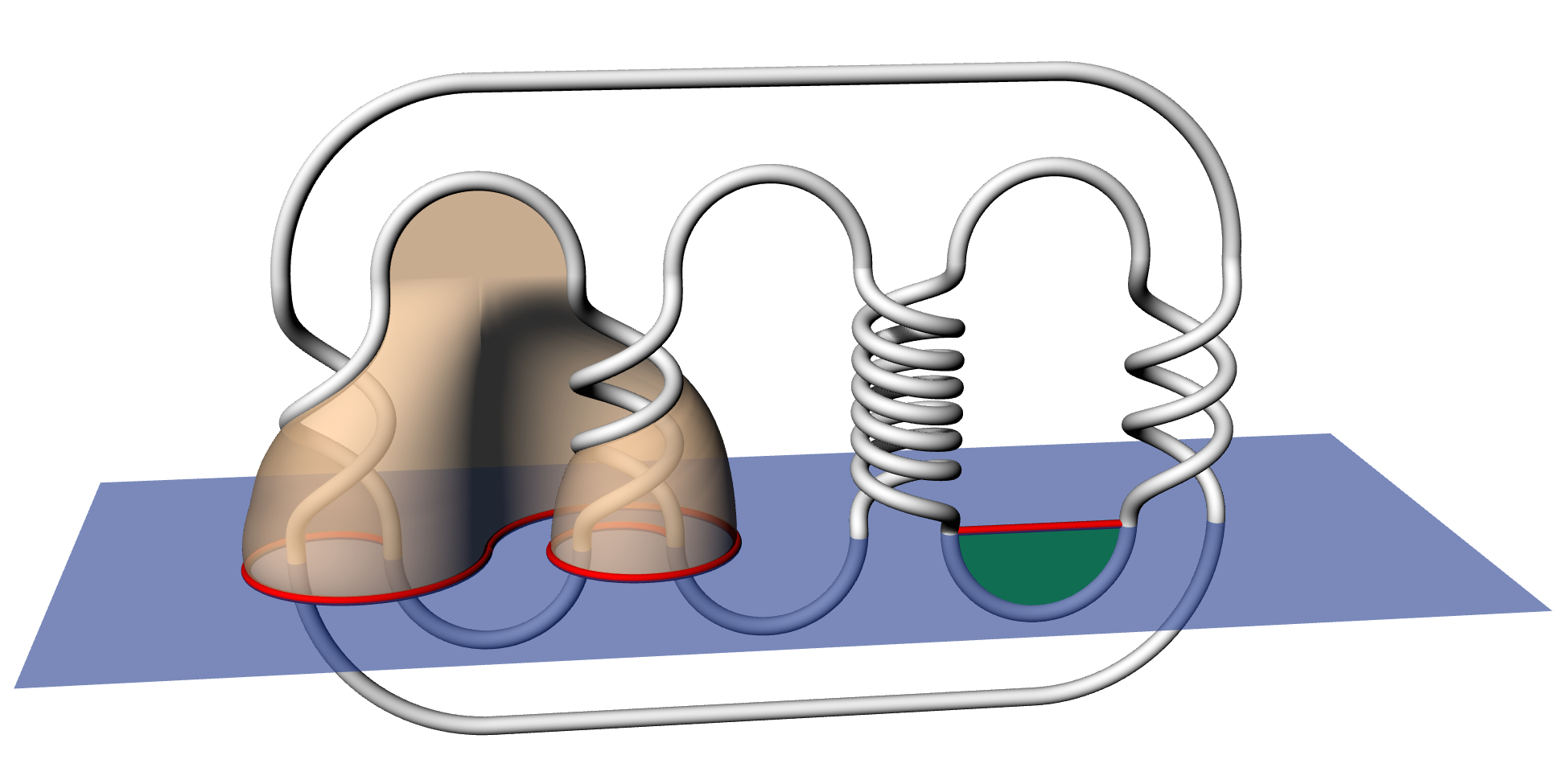}
\caption{A four branched pretzel knot with disjoint bridge disks on opposite sides of the bridge surface.}\label{fig:pretzeldisks}
\end{center}
\end{figure}
\end{proof}

This gives us that any pretzel knot $K_n=K_n(p_1,\ldots,p_n)$ with $\alpha\not = 1$ and $n\geq4$ has a stair-step bridge spectrum and is not high distance.  Thus, one could not find the bridge spectrum of $K_n$ with theorem \ref{thm:dist}.  As an example, $K_n(3p_1,3p_2,\ldots,3p_n)$ for $n\geq 4$ for any pretzel knot $K_n(p_1, \ldots,p_n)$ has ${\bf b}(K_n)=(n,n-1,\ldots,1,0)$ and distance one.

\end{document}